\documentclass{amsart}
\usepackage{pstricks-add}
\usepackage[utf8]{inputenc}
\newtheorem{theo}{Theorem}[section]
\newtheorem{lem}[theo]{Lemma}

\newtheorem{prop}[theo]{Proposition}

\theoremstyle{definition}
\newtheorem{defi}[theo]{Definition}

\theoremstyle{remark}
\newtheorem*{rem}{Remark}

\newtheorem*{eg*}{Example}

\def\mycaption{\caption}

\newcommand{\Dfn}[1]{\emph{#1}}
\newcommand{\Set}[1]{\ensuremath{\mathcal{#1}}}            
\DeclareMathOperator{\wex}{wex}
\DeclareMathOperator{\cro}{cro}
\newcommand{\qbin}[2]{\genfrac{[}{]}{0pt}{}{#1}{#2}_q}

\let\size=\abs

\title{Crossings, Motzkin paths and Moments}

\author{Matthieu Josuat-Verg\`es}

\address{LRI, CNRS and Universit\'e Paris-Sud, B\^atiment 490, 91405
  Orsay, France}

\author{Martin Rubey}

\address{Institut f\"ur Algebra, Zahlentheorie und Diskrete
  Mathematik, Leibniz Universit\"at Hannover, Welfengarten 1, 30167
  Hannover, Deutschland}

\dedicatory{Dedicated to Jean-Guy Penaud}

\begin{document}

\begin{abstract}
  Kasraoui, Stanton and Zeng, and Kim, Stanton and Zeng introduced
  certain $q$-analogues of Laguerre and Charlier polynomials.  The
  moments of these orthogonal polynomials have combinatorial models
  in terms of crossings in permutations and set partitions.  The aim
  of this article is to prove simple formulas for the moments of the
  $q$-Laguerre and the $q$-Charlier polynomials, in the style of the
  Touchard-Riordan formula (which gives the moments of some
  $q$-Hermite polynomials, and also the distribution of crossings in
  matchings).

  Our method mainly consists in the enumeration of weighted Motzkin
  paths, which are naturally associated with the moments. Some steps
  are bijective, in particular we describe a decomposition of paths
  which generalises a previous construction of Penaud for the case of
  the Touchard-Riordan formula.  There are also some non-bijective
  steps using basic hypergeometric series, and continued fractions
  or, alternatively, functional equations.
\end{abstract}

\maketitle
\section{Introduction}
\label{sec:introduction}

Our motivation is to derive in a uniform way generating functions for
matchings, set partitions and permutations refined by the number of
crossings.  We achieve this by enumerating certain weighted Motzkin
paths, which in turn prompt us to consider these counts as moments of
certain families of orthogonal polynomials.  In some cases, formulas
for these moments are already known.  However, the method of proof we
present in this algorithm is quite general, and leads to very simple
formulas.

Let us first define the notion of crossings in matchings, set
partitions and permutations.  To do so, it is best to draw the
objects we are interested in in a certain standard way.  We begin
with the set of matchings (or fixed-point free involutions) $\Set
M_{2n}$ of $\{1,\dots,2n\}$: these are drawn by putting the numbers
from $1$ to $2n$ in this order on a straight line, and then
connecting paired numbers by an arc.  Of course, arcs are always
drawn in a way such that any two arcs cross at most once, and no more
than two arcs intersect at any point, see the first picture in
Figure~\ref{fig:matching} for an example.  Then, a \Dfn{crossing in a
  matching} is, as one would expect, a pair of matched points
$\{i,j\}$ and $\{k,l\}$ with $i<k<j<l$, pictorially:
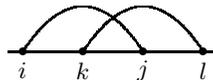
\begin{center}
{\setlength{\unitlength}{1mm}
\begin{picture}(28,10)(-30,-3)
\put(-30,0){\line(1,0){28}}
\put(-28,0){\circle*{1}}\put(-28,0){\makebox(0,-5)[c]{\small $i$}}
\put(-20,0){\circle*{1}}\put(-20,0){\makebox(0,-5)[c]{\small $k$}}
\put(-12,0){\circle*{1}}\put(-12,0){\makebox(0,-5)[c]{\small $j$}}
\put(-4,0){\circle*{1}}\put(-4,0){\makebox(0,-5)[c]{\small $l$}}
\qbezier(-28,0)(-20,12)(-12,0) \qbezier(-20,0)(-12,12)(-4,0)
\end{picture}}
\end{center}

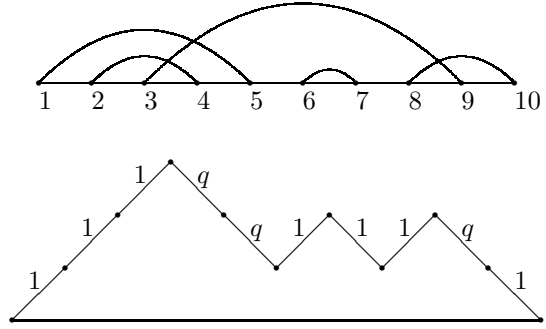
\begin{figure}
  \setlength{\unitlength}{20pt}
  \begin{picture}(10,7.5)(0,-4.5)
    \put(0,0.5){\line(1,0){9}}
    \put(0,0){\hbox{$1$}}\put(0,0.5){\circle*{0.1}}
    \put(1,0){\hbox{$2$}}\put(1,0.5){\circle*{0.1}}
    \put(2,0){\hbox{$3$}}\put(2,0.5){\circle*{0.1}}
    \put(3,0){\hbox{$4$}}\put(3,0.5){\circle*{0.1}}
    \put(4,0){\hbox{$5$}}\put(4,0.5){\circle*{0.1}}
    \put(5,0){\hbox{$6$}}\put(5,0.5){\circle*{0.1}}
    \put(6,0){\hbox{$7$}}\put(6,0.5){\circle*{0.1}}
    \put(7,0){\hbox{$8$}}\put(7,0.5){\circle*{0.1}}
    \put(8,0){\hbox{$9$}}\put(8,0.5){\circle*{0.1}}
    \put(9,0){\hbox{$10$}}\put(9,0.5){\circle*{0.1}}
    \qbezier(0,0.5)(2.0,2.5)(4,0.5)
    \qbezier(1,0.5)(2.0,1.5)(3,0.5)
    \qbezier(2,0.5)(5.0,3.5)(8,0.5)
    \qbezier(5,0.5)(5.5,1.0)(6,0.5)
    \qbezier(7,0.5)(8.0,1.5)(9,0.5)
    \put(-0.5,-4){\line(1,0){10}}
    \put(-0.5,-4){\circle*{0.1}}  \put(-0.5,-4){\line(1,1){1}} \put(-0.2,-3.4){$1$}
    \put(0.5,-3){\circle*{0.1}}   \put(0.5,-3){\line(1,1){1}}  \put(0.8,-2.4){$1$}
    \put(1.5,-2){\circle*{0.1}}   \put(1.5,-2){\line(1,1){1}}  \put(1.8,-1.4){$1$}
    \put(2.5,-1){\circle*{0.1}}   \put(2.5,-1){\line(1,-1){1}} \put(3.0,-1.4){$q$}
    \put(3.5,-2){\circle*{0.1}}   \put(3.5,-2){\line(1,-1){1}} \put(4.0,-2.4){$q$}
    \put(4.5,-3){\circle*{0.1}}   \put(4.5,-3){\line(1,1){1}}  \put(4.8,-2.4){$1$}
    \put(5.5,-2){\circle*{0.1}}   \put(5.5,-2){\line(1,-1){1}} \put(6.0,-2.4){$1$}
    \put(6.5,-3){\circle*{0.1}}   \put(6.5,-3){\line(1,1){1}}  \put(6.8,-2.4){$1$}
    \put(7.5,-2){\circle*{0.1}}   \put(7.5,-2){\line(1,-1){1}} \put(8.0,-2.4){$q$}
    \put(8.5,-3){\circle*{0.1}}   \put(8.5,-3){\line(1,-1){1}} \put(9.0,-3.4){$1$}
    \put(9.5,-4){\circle*{0.1}}   
  \end{picture}
  \caption{A matching of $\{1,\dots,10\}$ with $3$ crossings and the associated
    ``histoire de Hermite''.}
  \label{fig:matching}
\end{figure}

Indeed, the motivating example for this article is the
Touchard-Riordan formula, which gives, for each $n$, the generating
polynomial according to crossings for perfect matchings of the set
$\{1,\dots,2n\}$.  Denoting by $\cro(M)$ the number of crossings of
the matching $M$, we have:
\begin{theo}[Touchard~\cite{Tou52},
  Riordan~\cite{Rio75}]\label{th:Touchard}
  \begin{equation}\label{eq:Touchard}
    \sum_{M\in\Set M_{2n}} q^{\cro(M)}
    = \frac{1}{(1-q)^n}
    \sum_{k\geq0}(-1)^k\left(\binom{2n}{n-k}-\binom{2n}{n-k-1}\right)
    q^{\binom{k+1}{2}}.
  \end{equation}
\end{theo}
This has been proved in the 1950's by Touchard, although,
curiously, it seems that the formula was not given explicitly.  This
was later rectified by Riordan.

Quite similar to matchings, a set partition can be depicted by
connecting the numbers on the line which are in one block $B=\{b_1 <
b_2 < \dots < b_l\}$ by arcs $(b_1,b_2)$, $(b_2,b_3)$,\dots,
$(b_{l-1},b_l)$, see Figure~\ref{fig:setpartition} for an example.
Again, a \Dfn{crossing in a set partition} is what one would expect:
a pair of arcs $\{i,j\}$ and $\{k,l\}$ with $i<k<j<l$.  Denoting the
set of set partitions of $\{1,\dots, n\}$ by $\Pi_n$, the number of
crossings in a set partition $\pi$ by $\cro(\pi)$ and the number of
its blocks by $|\pi|$, we will obtain the following $q$-analogue of
the Stirling numbers of the second kind:
\begin{theo}\label{th:Charlier}
  \begin{align}\label{eq:Charlier}
    \sum_{ \substack{{ \pi\in\Pi_n }\\{|\pi|=k}} } q^{\cro(\pi)}
    = \frac1{(1-q)^{n-k}} \sum\limits_{j=0}^{k}
    \sum\limits_{i=j}^{n-k} (-1)^i \left(\tbinom n{k+i}\tbinom
      n{k-j}-\tbinom n{k+i+1}\tbinom n{k-j-1}\right) \qbin i j
    q^{\tbinom{j+1}2},
  \end{align}
  where $\qbin n k = \prod_{i=1}^k\frac{[n-k+i]_q}{[i]_q}$ is the
  $q$-binomial coefficient, and $[n]_q=1+q+\dots+q^{n-1}$.
\end{theo}

\begin{figure}
  \setlength{\unitlength}{20pt}
  \begin{picture}(7,6.5)(0,-4.5)
    \put(0,0.5){\line(1,0){7}}
    \put(0,0){\hbox{$1$}}\put(0,0.5){\circle*{0.1}}
    \put(1,0){\hbox{$2$}}\put(1,0.5){\circle*{0.1}}
    \put(2,0){\hbox{$3$}}\put(2,0.5){\circle*{0.1}}
    \put(3,0){\hbox{$4$}}\put(3,0.5){\circle*{0.1}}
    \put(4,0){\hbox{$5$}}\put(4,0.5){\circle*{0.1}}
    \put(5,0){\hbox{$6$}}\put(5,0.5){\circle*{0.1}}
    \put(6,0){\hbox{$7$}}\put(6,0.5){\circle*{0.1}}
    \put(7,0){\hbox{$8$}}\put(7,0.5){\circle*{0.1}}
    \qbezier(0,0.5)(2.0,2.5)(4,0.5)
    \qbezier(1,0.5)(3.0,2.5)(5,0.5)
    \qbezier(2,0.5)(2.5,1.0)(3,0.5)
    \qbezier(4,0.5)(5.5,2.5)(7,0.5)
    \put(-0.5,-4){\line(1,0){8}}
    \put(-0.5,-4){\circle*{0.1}}  \put(-0.5,-4){\line(1,1){1}}  \put(-0.2,-3.4){$y$}
    \put(0.5,-3){\circle*{0.1}}   \put(0.5,-3){\line(1,1){1}}   \put(0.8,-2.4){$y$}
    \put(1.5,-2){\circle*{0.1}}   \put(1.5,-2){\line(1,1){1}}   \put(1.8,-1.4){$y$}
    \put(2.5,-1){\circle*{0.1}}   \put(2.5,-1){\line(1,-1){1}}  \put(3.0,-1.4){$1$}
    \put(3.5,-2){\circle*{0.1}}   \put(3.5,-2){\line(1,0){1}}   \put(3.8,-1.8){$q$}
    \put(4.5,-2){\circle*{0.1}}   \put(4.5,-2){\line(1,-1){1}}  \put(5.0,-2.4){$q$}
    \put(5.5,-3){\circle*{0.1}}   \put(5.5,-3){\line(1,0){1}}   \put(5.8,-2.8){$y$}
    \put(6.5,-3){\circle*{0.1}}   \put(6.5,-3){\line(1,-1){1}}  \put(7.0,-3.4){$1$}
    \put(7.5,-4){\circle*{0.1}}   
  \end{picture}
  \caption{A set partition of $\{1,\dots,8\}$ into $4$ blocks with $2$
    crossings and the associated ``histoire de Charlier''.}
  \label{fig:setpartition}
\end{figure}
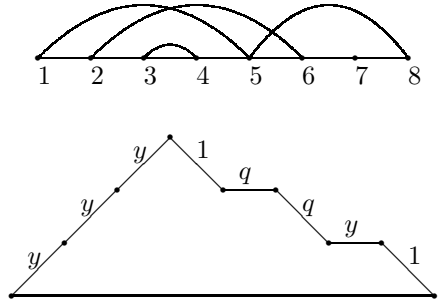

There is an alternative notion of crossings for set partitions, due
to Ehrenborg and Readdy~\cite{ER96}, coming from
juggling patterns.  Namely, we introduce an additional infinite arc
extending to the right from every maximal element of each block,
including singletons, see Figure~\ref{fig:setpartitionER} for an
example.  Denoting the number of crossings in such a drawing of a set
partition $\pi$ by $\cro^*(\pi)$, we have:
\begin{theo}[Gould~\cite{Gou61}]\label{th:Charlier*}
  \begin{equation}\label{eq:Charlier*}
    \sum_{\substack{{ \pi\in\Pi_n }\\{|\pi|=k}}} q^{\cro^*(\pi)}
    =\frac{1}{(1-q)^{n-k}} \sum\limits_{j=0}^{n-k}
    (-1)^j \tbinom n{k+j}\qbin{k+j}{j}.
  \end{equation}
\end{theo}
This is not a new result: essentially, this formula was already known
to Gould from another definition (the link with crossings is more recent
as will appear below).

\begin{figure}
  \setlength{\unitlength}{20pt}
  \begin{picture}(7,7.5)(0,-4.5)
    \put(0,0.5){\line(1,0){7}}
    \put(0,0){\hbox{$1$}}\put(0,0.5){\circle*{0.1}}
    \put(1,0){\hbox{$2$}}\put(1,0.5){\circle*{0.1}}
    \put(2,0){\hbox{$3$}}\put(2,0.5){\circle*{0.1}}
    \put(3,0){\hbox{$4$}}\put(3,0.5){\circle*{0.1}}
    \put(4,0){\hbox{$5$}}\put(4,0.5){\circle*{0.1}}
    \put(5,0){\hbox{$6$}}\put(5,0.5){\circle*{0.1}}
    \put(6,0){\hbox{$7$}}\put(6,0.5){\circle*{0.1}}
    \put(7,0){\hbox{$8$}}\put(7,0.5){\circle*{0.1}}
    \qbezier(0,0.5)(2.0,2.5)(4,0.5)
    \qbezier(1,0.5)(3.0,2.5)(5,0.5)
    \qbezier(2,0.5)(2.5,1.0)(3,0.5)
    \qbezier(4,0.5)(5.5,2.5)(7,0.5)
    \qbezier(3,0.5)(5.5,2.5)(8,2.5)
    \qbezier(5,0.5)(6.5,2.0)(8,2.0)
    \qbezier(6,0.5)(7.0,1.5)(8,1.5)
    \qbezier(7,0.5)(7.5,1.0)(8,1.0)
    \put(-0.5,-4){\line(1,0){8}}
    \put(-0.5,-4){\circle*{0.1}}  \put(-0.5,-4){\line(1,1){1}}  \put(-0.2,-3.4){$y$}
    \put(0.5,-3){\circle*{0.1}}   \put(0.5,-3){\line(1,1){1}}   \put(0.8,-2.4){$y$}
    \put(1.5,-2){\circle*{0.1}}   \put(1.5,-2){\line(1,1){1}}   \put(1.8,-1.4){$y$}
    \put(2.5,-1){\circle*{0.1}}   \put(2.5,-1){\line(1,-1){1}}  \put(3.0,-1.4){$q^2$}
    \put(3.5,-2){\circle*{0.1}}   \put(3.5,-2){\line(1,0){1}}   \put(3.8,-1.8){$q$}
    \put(4.5,-2){\circle*{0.1}}   \put(4.5,-2){\line(1,-1){1}}  \put(5.0,-2.4){$q^2$}
    \put(5.5,-3){\circle*{0.1}}   \put(5.5,-3){\line(1,0){1}}   \put(5.8,-2.8){$yq$}
    \put(6.5,-3){\circle*{0.1}}   \put(6.5,-3){\line(1,-1){1}}  \put(7.0,-3.4){$1$}
    \put(7.5,-4){\circle*{0.1}}   
  \end{picture}
  \caption{A set partition of $\{1,\dots,8\}$ into $4$ blocks with $6$
    crossings and the associated ``histoire de Charlier-$*$''.}
  \label{fig:setpartitionER}
\end{figure}
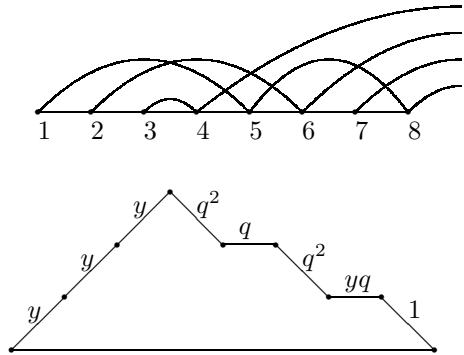

Finally, to depict a permutation $\sigma$, we connect the number $i$
with $\sigma(i)$ with an arc \emph{above} the line, if
$i\leq\sigma(i)$, otherwise with an arc \emph{below} the line, as
done in Figure~\ref{fig:permutation}.  The notion of crossing in a
permutation was introduced by Corteel~\cite{Cor07}, and is slightly
less straightforward: a pair of numbers
$(i,k)$ constitutes a \Dfn{crossing in a permutation}, if
$i<k\leq\sigma(i)<\sigma(k)$ or $\sigma(i)<\sigma(k)<i<k$:
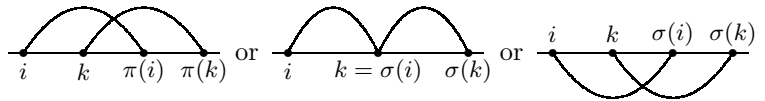
\begin{center}
{\setlength{\unitlength}{1mm}
\begin{picture}(29,14)(-30,-7)
\put(-30,0){\line(1,0){28}}
\put(-28,0){\circle*{1}}\put(-28,0){\makebox(0,-5)[c]{\small $i$}}
\put(-20,0){\circle*{1}}\put(-20,0){\makebox(0,-5)[c]{\small $k$}}
\put(-12,0){\circle*{1}}\put(-12,0){\makebox(0,-5)[c]{\small $\pi(i)$}}
\put(-4,0){\circle*{1}}\put(-4,0){\makebox(0,-5)[c]{\small $\pi(k)$}}
\qbezier(-28,0)(-20,12)(-12,0) \qbezier(-20,0)(-12,12)(-4,0)
\end{picture}
\begin{picture}(34,14)(-35,-7)
\put(-35,-1){or}
\put(-30,0){\line(1,0){28}}
\put(-28,0){\circle*{1}}\put(-28,0){\makebox(0,-5)[c]{\small $i$}}
\put(-16,0){\circle*{1}}\put(-16,0){\makebox(0,-5)[c]{\small $k=\sigma(i)$}}
\put(-4,0){\circle*{1}}\put(-4,0){\makebox(0,-5)[c]{\small $\sigma(k)$}}
\qbezier(-28,0)(-22,12)(-16,0) \qbezier(-16,0)(-10,12)(-4,0)
\end{picture}
\begin{picture}(34,14)(-35,-7)
\put(-35,-1){or}
\put(-30,0){\line(1,0){28}}
\put(-28,0){\circle*{1}}\put(-28,5){\makebox(0,-5)[c]{\small $i$}}
\put(-20,0){\circle*{1}}\put(-20,5){\makebox(0,-5)[c]{\small $k$}}
\put(-12,0){\circle*{1}}\put(-12,5){\makebox(0,-5)[c]{\small $\sigma(i)$}}
\put(-4,0){\circle*{1}}\put(-4,5){\makebox(0,-5)[c]{\small $\sigma(k)$}}
\qbezier(-28,0)(-20,-12)(-12,0) \qbezier(-20,0)(-12,-12)(-4,0)
\put(0,-1){.}
\end{picture}}
\end{center}

Denoting the set of permutations of $\{1,\dots, n\}$ by
$\mathfrak{S}_n$, and the number of weak exceedances, {\it i.e.}
numbers $i$ with $\sigma(i)\geq i$, of a permutation $\sigma$ by
$\wex(\sigma)$, we have:
\begin{theo}[Josuat-Verg\`es~\cite{MJV}, Corteel, Josuat-Verg\`es,
  Prellberg, Rubey~\cite{CJPR09}]
  \label{th:Laguerre}
  \begin{equation}\label{eq:Laguerre}
    \begin{split}
      &\sum_{\sigma\in\mathfrak{S}_n}
      y^{\wex(\sigma)} q^{\cro(\sigma)} \\
      &= \frac 1{(1-q)^n} \sum\limits_{k=0}^n (-1)^k 
      \left(\sum\limits_{j=0}^{n-k} y^j
        \Big( \tbinom{n}{j}\tbinom{n}{j+k} -
        \tbinom{n}{j-1}\tbinom{n}{j+k+1}\Big) \right) 
      \left(\sum\limits_{i=0}^k y^iq^{i(k+1-i)}  \right).
    \end{split}
  \end{equation}
\end{theo}
This theorem recently found a rather different proof by the first
author~\cite{MJV}.  In the present article we provide an alternative,
using a bijective decomposition of weighted Motzkin paths that gives
a natural interpretation for the two inner sums.

The rest of this article is organised as follows.  In Section
\ref{sec:orth-poly}, we present some background material concerning
the combinatorial theory of orthogonal polynomials.  In
Section~\ref{sec:penaud-decomp}, we describe the decomposition of
weighted Motzkin paths mentioned above, in full generality.  Each
Motzkin path will be decomposed into a Motzkin prefix and another
Motzkin path satisfying certain additional conditions.  In
Section~\ref{sec:prefixes}, we enumerate Motzkin prefixes, and in
Section~\ref{sec:paths2} the other set of paths appearing in the
decomposition are enumerated.

There are three appendices.  In the first appendix we give an
alternative point of view of the decomposition presented in
Section~\ref{sec:penaud-decomp}, using inverse relations.  In the
second appendix, we give a bijective proof of the formula for the
generating function of the paths appearing in the decomposition in
the case of set-partitions, using a sign-reversing involution.  It is
thus possible to give a fully bijective proof of
Theorem~\ref{th:Charlier}, analogous to Penaud's proof of the
Touchard-Riordan formula.  Finally, in the last appendix we sketch a
proof showing that one cannot expect closed forms for Motzkin
prefixes with weights different from those considered in
Section~\ref{sec:prefixes}.

\begin{figure}
  \setlength{\unitlength}{20pt}
  \begin{picture}(7,6)(0,-4.5)
    \put(0,0.5){\line(1,0){7}}
    \put(0,0){\hbox{$1$}}\put(0,0.5){\circle*{0.1}}
    \put(1,0){\hbox{$2$}}\put(1,0.5){\circle*{0.1}}
    \put(2,0){\hbox{$3$}}\put(2,0.5){\circle*{0.1}}
    \put(3,0){\hbox{$4$}}\put(3,0.5){\circle*{0.1}}
    \put(4,0){\hbox{$5$}}\put(4,0.5){\circle*{0.1}}
    \put(5,0){\hbox{$6$}}\put(5,0.5){\circle*{0.1}}
    \put(6,0){\hbox{$7$}}\put(6,0.5){\circle*{0.1}}
    \put(7,0){\hbox{$8$}}\put(7,0.5){\circle*{0.1}}
    \qbezier(0,0.5)(1.0,1.5)(2,0.5)
    \qbezier(1,0.5)(2.0,1.5)(3,0.5)
    \qbezier(2,0.5)(4.0,2.5)(6,0.5)
    \put(4,0.75){\circle{0.5}}
    \qbezier(6,0.5)(6.5,1.0)(7,0.5)
    \qbezier(7,0.5)(6.0,-1.5)(5,0.5)
    \qbezier(5,0.5)(3.0,-2.0)(1,0.5)
    \qbezier(3,0.5)(1.5,-1.5)(0,0.5)
    \put(-0.5,-4){\line(1,0){8}}  
    \put(-0.5,-4){\circle*{0.1}}  \put(-0.5,-4){\line(1,1){1}}  \put(-0.2,-3.4){$y$}
    \put(0.5,-3){\circle*{0.1}}   \put(0.5,-3){\line(1,1){1}}   \put(0.6,-2.4){$yq$}
    \put(1.5,-2){\circle*{0.1}}   \put(1.5,-2){\line(1,0){1}}   \put(1.8,-1.8){$yq^2$}
    \put(2.5,-2){\circle*{0.1}}   \put(2.5,-2){\line(1,-1){1}}  \put(3.0,-2.4){$q$}
    \put(3.5,-3){\circle*{0.1}}   \put(3.5,-3){\line(1,0){1}}   \put(3.8,-2.8){$y$}
    \put(4.5,-3){\circle*{0.1}}   \put(4.5,-3){\line(1,0){1}}   \put(4.8,-2.8){$1$}
    \put(5.5,-3){\circle*{0.1}}   \put(5.5,-3){\line(1,0){1}}   \put(5.8,-2.8){$yq$}
    \put(6.5,-3){\circle*{0.1}}   \put(6.5,-3){\line(1,-1){1}}  \put(7.0,-3.4){$1$}
    \put(7.5,-4){\circle*{0.1}}   
  \end{picture}
  \caption{A permutation of $\{1,\dots,8\}$ with $5$ weak exceedances and $5$
    crossings and the associated ``histoire de Laguerre''.}
  \label{fig:permutation}
\end{figure}
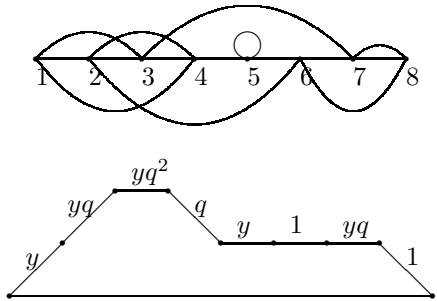

\section{Orthogonal Polynomials, moments and histoires}
\label{sec:orth-poly}

Motzkin paths are at the heart of the \emph{combinatorial theory of
  orthogonal polynomials}, as developed by
Flajolet~\cite{Fla82} and Viennot~\cite{Vie84}.  This theory
tells us, that the moments of any family of orthogonal polynomials
are given by a certain weighted count of Motzkin paths.

More precisely, by Favard's theorem, any monic sequence of orthogonal
polynomials $(P_n)_{n\geq0}$ satisfies a three term recurrence of the form
$$ xP_n(x) = P_{n+1}(x)+ b_n P_n(x) + \lambda_n P_{n-1}(x),$$
where $b_n$ and $\lambda_n$ do not depend on $x$.
Given this recurrence, the $n$\textsuperscript{th} moment $\mu^P_n$
of $P$ can be expressed as the weighted sum of Motzkin paths of
length $n$, that is, paths taking up ($\nearrow$), down ($\searrow$)
and level ($\rightarrow$) steps, starting and ending at height $0$,
and not going below this height, where a horizontal step at height
$h$ has weight $b_h$ and a down step starting at height $h$ has
weight $\lambda_h$.

\subsection{Histoires}
\label{sec:histoires}
Three basic examples of families of orthogonal polynomials are given
by (rescalings of) the Hermite, Charlier and Laguerre polynomials,
where the moments count matchings ($b_n=0$, $\lambda_n=n$),
set partitions ($b_n=1+n$, $\lambda_n=n$) and permutations
($b_n=2n+1$, $\lambda_n=n^2$) respectively.  It turns out that the
Hermite, Charlier and Laguerre polynomials indeed have beautiful
$q$-analogues such that the moments count the corresponding objects,
and $q$ marks the number of crossings.  We want to establish this
correspondence via ``histoires'':

\begin{defi}
  Consider a family of orthogonal polynomials with coefficients $b_n$
  and $\lambda_n$, and fix $a_n$ and $c_n$ such that
  $\lambda_n=a_{n-1} c_n$ for all $n$. Suppose that for every fixed
  $n$, the coefficients $a_n$, $b_n$ and $c_n$ are polynomials
  such that each monomial has coefficient 1 (as will appear shortly,
  this is general enough in our context).

  We then call a weighted Motzkin path \Dfn{histoire}, when the
  weight of an up step $\nearrow$ (respectively a level step
  $\rightarrow$ or a a down step $\searrow$) starting at level $h$ is
  one of the monomials appearing in $a_h$ (respectively $b_h$ or
  $c_h$).
\end{defi}

We want to consider four different families of ``histoires'',
corresponding to $q$-analogues of the Hermite, Charlier and Laguerre
polynomials.

\begin{prop}
  There are weight-preserving bijections between 
  \begin{itemize}
  \item matchings $M$ with weight $q^{\cro(M)}$, and ``histoires
    de Hermite''\ defined by $b_n=0$, $a_n=1$ and $c_n=[n]_q$,
  \item set partitions $\pi$ with weight $y^{|\pi|}q^{\cro(\pi)}$,
    and ``histoires de Charlier''\ defined by $b_n=y+[n]_q$,
    $a_n=y$ and $c_n=[n]_q$,
  \item set partitions $\pi$ with weight
    $y^{|\pi|}q^{\cro^*(\pi)}$, and ``histoires de Charlier-$*$''\
    defined by $b_n=yq^n+[n]_q$, $a_n=yq^n$ and $c_n=[n]_q$, and
  \item permutations $\sigma$ with weight
    $y^{\wex(\sigma)}q^{\cro(\sigma)}$, and ``histoires de
    Laguerre''\ defined by $b_n=yq^n+[n]_q$, $a_n=yq^n$ and
    $c_n=[n]_q$.
  \end{itemize}
\end{prop}
These bijections are straightforward modifications of classical
bijections used by Viennot~\cite{Vie84}.  We detail them here for
convenience, but also because of their beauty\dots\ Examples can be
found in Figures~\ref{fig:matching}--\ref{fig:permutation}.
\begin{proof}
  The bijection connecting matchings and ``histoires de Hermite'',
  such that crossings are recorded in the exponent of $q$, goes as
  follows: we traverse the matching, depicted in the standard way,
  from left to right, while we build up the Motzkin path step by
  step, also from left to right.  For every arc connecting $i$ and
  $j$ with $i<j$, we call $i$ an \Dfn{opener} and $j$ a \Dfn{closer}.
  When we have traversed the matching up to and including number
  $\ell$, we call the openers $i\leq\ell$ with corresponding closers
  $j<\ell$ \Dfn{active}.  Openers are translated into up steps with
  weight $1$.  Accordingly, when we encounter a closers $\ell$ it
  becomes a down step with weight $q^k$, where $k$ is the number of
  active openers between $\ell$ and the opener corresponding to
  $\ell$.  It is a enjoyable exercise to see that this is indeed a
  bijection, and that a matching with $k$ crossings corresponds to a
  Dyck path of weight $q^k$.

  The bijection between set partitions and ``histoires de Charlier'',
  due to Anisse Kasraoui and Jiang Zeng~\cite{KaZe06}, is very
  similar: in addition to openers and closers, which are the
  non-maximal and non-minimal elements of the blocks of the set
  partition, we now also have \Dfn{singletons}, which are neither
  openers nor closers.  Elements that are openers and closers at the
  same time are called \Dfn{transients}.  Non-transient openers are
  translated into up steps with weight $y$, and singletons are
  translated into level steps with weight $y$.  Non-transient closers
  $\ell$ are translated into down steps both with weight $q^k$, where
  $k$ is the number of active openers between $\ell$ and the opener
  corresponding to $\ell$.  Finally, transient closers $\ell$ become
  level-steps with weight $q^k$, with $k$ as before.

  To obtain a ``histoire de Charlier-$*$'' of a set partition, using
  the modified definition of crossings, we only have to multiply the
  weights of steps corresponding to closers and singletons by $q^k$,
  where $k$ is the number of crossings of the infinite arc with other
  arcs.

  It remains to describe the bijection between permutations and
  ``histoires de Laguerre'', due to Dominique Foata and Doron
  Zeilberger, which is usually done in a different way than in what
  follows, however.  To obtain the Motzkin path itself, we ignore all
  the arcs below the line and also the loops corresponding to fixed
  points.  What remains can be interpreted as a set partition, and
  thus determines a Motzkin path.  Moreover, the weights of the down
  steps are computed as in the case of set partitions, except that
  the weight of each of those steps needs to be multiplied by $y$.
  The weights of the level steps that correspond to transients of the
  set partition are also computed as before, but are then multiplied
  by $yq$.  Level steps that correspond to fixed points of the
  permutation get weight $y$.  The weights of the remaining steps are
  computed by deleting all arcs above the line, and again
  interpreting what remains as a set partition.  However, this set
  partition has to be traversed from right to left, and weights are
  accordingly put onto the up steps of the Motzkin path.  Later, it
  will be more convenient to move the factor $y$ that appears in the
  weight of all the down steps onto the weight of the corresponding
  up steps, see Figure~\ref{fig:permutation} for an example.
\end{proof}

\subsection{Particular classes of orthogonal polynomials}
\label{sec:ortho}
In this section we relate the families of orthogonal polynomials
introduced via their parameters $b_n$ and $\lambda_n$ in
Section~\ref{sec:histoires} to classical families.  We follow the
Askey-Wilson scheme~\cite{KoSw98} for their definition.

The \Dfn{continuous $q$-Hermite polynomials} $H_n=H_n(x|q)$ can be
defined~\cite[Section~3.26]{KoSw98} by the recurrence relation
\begin{equation*}
  2 x H_n = H_{n+1} + (1-q^n) H_{n-1},
\end{equation*}
with $H_0=1$.

\begin{theo}[Ismail, Stanton and Viennot~\cite{ISV87}]
  Define rescaled continuous $q$-Hermite polynomials $\tilde
  H_n=\tilde H_n(x|q)$ as
  \begin{equation}
    \tilde H_n(x|q)=(1-q)^{-n/2}H_n(x\tfrac{\sqrt{1-q}}{2}|q).
  \end{equation}
  They satisfy the recurrence relation
  \begin{equation}
    x \tilde H_n = \tilde H_{n+1} + [n]_q \tilde H_{n-1},
  \end{equation}
  and their even moments are given by
  \begin{equation}\label{eq:muTouchard}
    \mu^{\tilde H}_{2n} = \sum_{M\in\Set M_{2n}} q^{\cro(M)}.
  \end{equation}
  The odd moments are all zero.
\end{theo}

The \Dfn{Al-Salam-Chihara} polynomials $Q_n=Q_n(x;a,b|q)$ can be
defined~\cite[Section~3.8]{KoSw98} by the recurrence relation
\begin{equation*}
  2 x Q_n = Q_{n+1} + (a + b)q^n Q_n + (1-q^n)(1-abq^{n-1})Q_{n-1},
\end{equation*}
with $Q_0=1$.
We consider two different specialisations of these polynomials.  The
first was introduced by Kim, Stanton and 
Zeng~\cite{KSZ06}, and in their Proposition~5 they also gave a formula for the
moments.  However, the formula that follows from our
Theorem~\ref{th:Charlier} appears to be much simpler.
\begin{theo}[Kim, Stanton, Zeng \cite{KSZ06}]
  Define $q$-Charlier polynomials $\tilde C_n=\tilde C_n(x;y|q)$ as
  \begin{equation}
    \tilde C_n(x;y|q)=\left(\tfrac{y}{1-q}\right)^{n/2}
    Q_n\left(\sqrt{\tfrac{1-q}{4y}}\left(x-y-\tfrac{1}{1-q}\right);
      \tfrac{-1}{\sqrt{y(1-q)}}, 0 \;\big\vert\; q\right).
  \end{equation}
  They satisfy the recurrence relation
  \begin{equation}
    x \tilde C_n = \tilde C_{n+1} + ( y + [n]_q) \tilde C_n + y [n]_q
    \tilde C_{n-1}
  \end{equation}
  and their moments are given by
  \begin{equation}
    \mu^{\tilde C}_{n} = \sum_{\pi\in\Pi_n} y^{|\pi|}q^{\cro(\pi)}.
  \end{equation}
\end{theo}

The other specialisation was introduced by Kasraoui, Stanton and
Zeng~\cite{KSZ08}, however, without providing a formula for the
moments (these are actually a particular case of octabasic
$q$-Laguerre polynomials from \cite{SS}).

\begin{theo}[Kasraoui, Stanton, Zeng \cite{KSZ08}]
  Define $q$-Laguerre polynomials $\tilde L_n=\tilde L_n(x;y|q)$ as
  \begin{equation}
    \tilde L_n(x; y|q) = \left(\tfrac{\sqrt{y}}{q-1}\right)^n
    Q_n\left(\tfrac{(q-1)x+y+1}{2\sqrt{y}}; \tfrac{1}{\sqrt{y}}, \sqrt{y} q \;\big\vert\; q\right).
  \end{equation}
  They satisfy the recurrence relation:
  \begin{equation*}
    x \tilde L_n = \tilde L_{n+1} + ([n]_q + y[n+1]_q) \tilde L_n + y
    [n]_q^2 \tilde L_{n-1}.
  \end{equation*}
  and their moments are given by
  \begin{align}\label{eq:muLaguerre}
    \mu^{\tilde L}_{n} = \sum_{\sigma\in\mathfrak{S}_n}
    y^{\wex(\sigma)} q^{\cro(\sigma)}.
  \end{align}
\end{theo}

The \Dfn{Al-Salam-Carlitz I} polynomials $U^{(a)}_n(x|q)$ can be
defined~\cite[Section~3.24]{KoSw98} by the recurrence relation
\begin{equation}
  xU_{n}^{(a)}(x|q) = U_{n+1}^{(a)}(x|q) + (a+1)q^nU_{n}^{(a)}(x|q) 
  -  q^{n-1} a(1-q^n)U_{n-1}^{(a)}(x|q),
\end{equation}
with $U_{0}^{(a)}(x|q)=1$.

\begin{theo}[de Médicis, Stanton, White~\cite{MSW}]
  Define modified $q$-Charlier polynomials $\tilde C_n^*=\tilde
  C_n^*(x;y|q) $ as:
  \begin{equation}
    C^*_n(x;y|q) =  y^n U_n^{\left(\frac{-1}{y(1-q)}\right)}
    \left(   \tfrac xy - \tfrac 1{y(1-q)} | q  \right).
  \end{equation}
  They satisfy the recurrence relation
  \begin{equation}
    x\tilde C^*_n = \tilde C^*_{n+1} 
    +  ( yq^n + \left[n\right]_q ) \tilde C^*_n 
    + y\left[n\right]_q q^{n-1} \tilde C^*_{n-1},
  \end{equation}
  and their moments are given by
  \begin{align}\label{eq:muCharlier*}
    \mu^{\tilde C^*}_{n} = \sum_{\pi\in\Pi_n }y^{|\pi|}
    q^{\cro^*(\pi)}.
  \end{align}
\end{theo}
The result from \cite{MSW} was actually stated with another statistic, but both
correspond to Carlitz' $q$-analogue of the Stirling numbers of the second kind $S[n,k]$,
which are such that $S[n,k]=S[n-1,k-1]+[k]_qS[n-1,k]$, and 
\[\mu^{\tilde C^*}_{n} = \sum_{k=1}^{n} S[n,k]y^k.\]

\section{Penaud's decomposition}
\label{sec:penaud-decomp}

Let us first briefly recall Penaud's strategy to prove the
Touchard-Riordan formula for the moments of the rescaled continuous
$q$-Hermite polynomials $\tilde H_n$.  As already indicated in the
introduction, his starting point was their combinatorial
interpretation in terms of weighted Dyck paths, down steps starting
at level $h\geq1$ having weight $[h]_q$, up steps having weight $1$.

As the total number of down steps in these paths is $n$, we may take
out a factor $(1-q)^{-n}$, and instead consider paths with down steps
having weight $1-q^h$, or, equivalently, consider paths with down
steps having weight $1$ \emph{or} $-q^h$.

The next step is to (bijectively) decompose each path into two
objects: the first is a left factor of an unweighted Dyck path of
length $n$ and final height $n-2k\geq0$, for some $k$.  The second
object, in some sense the remainder, is a weighted Dyck path of
length $k$ with the same possibilities for the weights as in the
original path, except that peaks (consisting of an up step
immediately followed by a down step) of weight $1$ are not allowed.
This decomposition will be generalised in
Lemma~\ref{lem:decomposition} below.

The left factors are straightforward to count, the result being the
ballot numbers $\binom{2n}{n-k}-\binom{2n}{n-k-1}$.  For the
remainders, Penaud presented a bijective proof that the sum
of their weights is given by $(-1)^k q^{\binom{k+1}{2}}$.  Summing
over all $k$ we obtain the Touchard-Riordan
formula~\eqref{eq:Touchard}.

\subsection{The general setting}
\label{sec:general-setting}
\begin{defi}\label{defsMP} 
  Let $\mathcal{M}_n(a,b,c,d;q)$ be the set of weighted Motzkin paths
  of length $n$, such that the weight of
\begin{itemize}
\item an up step $\nearrow$ starting at level $h$ is either $1$ or
  $-q^{h+1}$,
\item of a level step $\rightarrow$ starting at level $h$ is either
  $d$ or $(a+b)q^h$,
\item a down step $\searrow$ starting at level $h$ is either $c$ or
  $-abq^{h-1}$.
\end{itemize}
Furthermore, let
$\mathcal{M}^*_n(a,b,c;q)\subset\mathcal{M}_n(a,b,c,d;q)$ be the
subset of paths that do not contain any
\begin{itemize}
\item level step $\rightarrow$ of weight $d$,
\item peak $\nearrow\searrow$ such that the up step has weight $1$
  and the down step has weight $c$.
\end{itemize}
Finally, let $\mathcal{P}_{n,k}(c,d)$ be the set of left factors of
Motzkin paths of length $n$ and final height $k$, such that the
weight of
\begin{itemize}
\item an up step $\nearrow$ is $1$,
\item a level step $\rightarrow$ is $d$,
\item a down step $\searrow$ is $c$.
\end{itemize}
\end{defi}

With these definitions, the decomposition used by Penaud can
be generalised in a natural way as follows:
\begin{lem}
  \label{lem:decomposition}
  There is a bijection $\Delta$ between $\mathcal{M}_n(a,b,c,d;q)$
  and the disjoint union of the sets
  $\mathcal{P}_{n,k}(c,d)\times\mathcal{M}^*_k(a,b,c;q)$ for
  $k\in\{0,\dots,n\}$.
\end{lem}
\begin{proof} 
  Let $H$ be a path in $\mathcal{M}_n(a,b,c,d;q)$.  Consider the
  maximal factors $f_1,\dots,f_j$ of $H$ that are Motzkin paths and
  have up steps of weight $1$, level steps of weight $d$ and down
  steps of weight $c$.  We can thus factorise $H$ as
  $h_0f_1h_1f_2\dots f_jh_j$.

  Since this factorisation is uniquely determined, we can define
  $\Delta(H)=(H_1,H_2)$ as follows:
  $$
  H_1 = (\nearrow)^{|h_0|} f_1 (\nearrow)^{|h_1|} f_2 \dots f_j (\nearrow)^{|h_j|}
  \qquad\hbox{ and } \qquad
  H_2 = h_0\dots h_j.
  $$

  Thus, $H_1$ is obtained from $H$ by replacing each step in the
  $h_i$ by an up step $\nearrow$, and $H_2$ is obtained from $H$ by
  deleting the factors $f_i$.  Since the $f_i$ are Motzkin paths, the
  weight of $H$ is just the product of the weights of $H_1$ and
  $H_2$.  Furthermore, it is clear that $H_1$ is a path in
  $\mathcal{P}_{n,k}(c,d)$ with final height
  $k=|h_0|+|h_1|+\dots+|h_j|$.  We observe that the $h_i$ cannot
  contain a level step $\rightarrow$ of weight $d$ or a peak
  $\nearrow\searrow$ such that the up step has weight $1$ and the
  down step has weight $c$, because then the factorisation of $H$
  would not have been complete.  Thus $H_2$ is a path in
  $\mathcal{M}^*_k(a,b,c;q)$.

  It remains to verify that $\Delta$ is indeed a bijection.  To do
  so, we describe the inverse map: let
  $(H_1,H_2)\in\mathcal{P}_{n,k}(c,d)\times \mathcal{M}^*_k(a,b,c;q)$
  for some $k\in\{0,\dots,n\}$.  Thus, there exists a unique
  factorisation 
  $$
  H_1=(\nearrow)^{u_0} f_1 (\nearrow)^{u_1} f_2 \dots f_j
  (\nearrow)^{u_j}
  $$
  such that the $f_i$ are Motzkin paths and $k=\sum_{\ell=0}^j
  u_\ell$.  Write $H_2$ as $h_0\dots h_j$, where the factor $h_\ell$
  has length $u_\ell$.  Then $\Delta^{-1}(H_1,H_2) =
  h_0f_1h_1f_2\dots f_jh_j$ is the preimage of $(H_1,H_2)$.
\end{proof}

\subsection{Specialising to matchings, set partitions and
  permutations}\label{sec:specialising}
As remarked in the introduction of this section, we begin by
multiplying the weighted sum of all Motzkin paths by an appropriate
power of $1-q$.  In the case of \emph{matchings} of $\{1,\dots,2n\}$, we
are in fact considering Dyck paths of length $2n$ where a down step
starting at height $h$ has weight $[h]_q$.  Multiplying the weighted
sum with $(1-q)^n$, or, equivalently, multiplying the weight of each
down step by $1-q$, we thus obtain Dyck paths having down steps
starting at height $h$ weighted by $1-q^h$, which fits well into the
model introduced in Definition~\ref{defsMP}: namely, the set
$\mathcal{M}_n(a,b,c,d;q)$ with $a=0$, $b=0$, $c=1$ and $d=0$
consists precisely of these paths -- except that they are all
reversed.

In the case of \emph{set partitions} of $\{1,\dots,n\}$,
multiplying the weighted sum by $(1-q)^n$ and reversing all paths we
see that we need to enumerate the set $\mathcal{M}_n(a,b,c,d;q)$ with
$a=0$, $b=-1$, $c=y(1-q)$ and $d=1+y(1-q)$.  When using the
\emph{modified} definition of crossings in \emph{set partitions}, we
obtain surprisingly different parameters, namely $a=-1$, $b=y(1-q)$,
$c=0$ and $d=1$.  Finally, the case of \emph{permutations} of
$\{1,\dots,n\}$ is covered by enumerating the set
$\mathcal{M}_n(a,b,1,d;q)$ with $a=-1$, $b=-yq$, $c=y$ and $d=1+y$.


\section{Counting $\mathcal{P}_{n,k}(c,d)$}
\label{sec:prefixes}

In general, formulas for the cardinality of $\mathcal{P}_{n,k}(c,d)$
can be found easily using Lagrange inversion \cite{Sta}. Consider the
generating function $P_k = \sum_n \size{\mathcal{P}_{n,k}(c,d)} t^n$,
we want to determine the coefficient of $t^{n+1}$ in $tP_k=(t
P_0)^{k+1}$.  Observing the relationship
$$
t P_0 = t\left( 1 + d (t P_0) + c (t P_0)^2\right)
$$
we find that %
$[t^n](t P_0)^k = \frac{k}{n}[z^{n-k}](1 + d z + c z^2)^n$, and thus

\begin{equation}
  \label{eq:trinomial}
  \size{\mathcal{P}_{n,k}(c,d)}=
  \frac{k+1}{n+1}\sum_{l=0}^{n-k}\binom{n+1}{l}\binom{l}{2l-n+k}d^{2l-n+k}c^{n-k-l}.
\end{equation}

To count matchings, set partitions or permutations according to
crossings (modified or not), the only sets of parameters that we need
to consider are $(c,d)=(1,0)$, $(c,d)=(1,2)$ and $(c,d)=(0,1)$.
Curiously, these are precisely the values for which
Equation~\eqref{eq:trinomial} allows a closed form, {\it i.e.} can be
written as a linear combination of hypergeometric terms.  A (sketch
of a) justification of this fact is given in
Appendix~\ref{sec:trinomial}.

\subsection{Matchings}
For matchings, we have $(c,d)=(1,0)$ and we obtain the ballot
numbers:
\begin{lem}\label{dyck_pre} 
  The cardinality of $\mathcal{P}_{n,n-2k}(1,0)$, {\it i.e.} the number of
  left factors of Dyck paths of length $n$ and final height $n-2k\geq
  0$ is
\[
  \binom nk - \binom n{k-1}.
\]
\end{lem}

\subsection{Set partitions and permutations}
For set partitions and permutations, we have $(c,d)=(y,1+y)$ and
obtain the following:
\begin{lem}\label{motz_pre}
  The generating function for $\mathcal{P}_{n,k}(y,1+y)$ is:
  \begin{equation}\label{motz_pre_eq}
    \sum\limits_{j=0}^{n-k} 
    \left( \binom{n}{j}\binom{n}{j+k} -
      \binom{n}{j-1}\binom{n}{j+k+1}\right) y^j.
  \end{equation}
\end{lem}
\begin{proof}
  The elements of $\mathcal{P}_{n,k}(y,1+y)$ have weight $1+y$ on
  each level step.  However, it is again more convenient to pretend
  that there are two different kinds of level steps, with weight $1$
  and $y$ respectively.  Let $P$ be a left factor of a Motzkin path
  with weight $y^j$.  We then use the following step by step
  translation to transform it into a pair $(C_1,C_2)$ of
  non-intersecting paths taking north and east steps, starting at
  $(0,1)$ and $(1,0)$ respectively (see Figure~\ref{paths_r} for an
  example):
  \begin{center} 
    \begin{tabular}[h]{l|c|c}
      $i$\textsuperscript{th} step of $P$ &
      $i$\textsuperscript{th} step of $C_1$ &
      $i$\textsuperscript{th} step of $C_2$ \\\hline
      $\nearrow$                & $\uparrow$    & $\rightarrow$ \\
      $\rightarrow$, weight $1$ & $\uparrow$    & $\uparrow$ \\
      $\rightarrow$, weight $y$ & $\rightarrow$ & $\rightarrow$ \\
      $\searrow$, weight $y$    & $\rightarrow$ & $\uparrow$
    \end{tabular}
  \end{center}

  \begin{figure}[h!tp] \center \psset{unit=5mm}
  \begin{pspicture}(0,-1)(8,4)
  \psgrid[gridcolor=gray,griddots=4,subgriddiv=0,gridlabels=0](0,0)(8,4)
  \psline(0,0)(4,0)(5,1)(6,1)(7,1)(8,2)
  \rput(0.5,0.5){$1$}
  \rput(1.5,0.5){$y$}
  \rput(2.5,0.5){$y$}
  \rput(3.5,0.5){$1$}
  \rput(4.5,1.5){$1$}
  \rput(5.5,1.5){$y$}
  \rput(6.5,1.5){$1$}
  \rput(7.5,2.5){$1$}
  \end{pspicture} \psset{unit=3mm}
  \hspace{2.5cm}
  \begin{pspicture}(-2,-1)(9,7)
  \psline{->}(0,0)(7,0) \psline{->}(0,0)(0,7)
  \psline(0,0)(7,0)  \psline(0,0)(0,7)
  \psline(0,1)(7,1)  \psline(1,0)(1,7)
  \psline(0,2)(7,2)  \psline(2,0)(2,7)
  \psline(0,3)(7,3)  \psline(3,0)(3,7)
  \psline(0,4)(7,4)  \psline(4,0)(4,7)
  \psline(0,5)(7,5)  \psline(5,0)(5,7)
  \psline(0,6)(7,6)  \psline(6,0)(6,7)
  \psdots(1,0)(0,1)(3,6)(6,3)
  \psline[linewidth=0.8mm](1,0)(1,1)(3,1)(3,2)(5,2)(5,3)(6,3)
  \psline[linewidth=0.8mm](0,1)(0,2)(2,2)(2,4)(3,4)(3,6)
  \rput(-1,-1){\small 0}\rput(1,-1){\small 1}\rput(3,-1){\small$j$}\rput(8,-1){\small$j+k+1$}
  \rput(-1,1){\small 1}\rput(-3,3){\small$n-k-j$}\rput(-3,6){\small$n-j+1$}
  \end{pspicture}
  \mycaption{A bijection to count $\mathcal{P}_{n,k}(y,1+y)$.
  \label{paths_r}}
  \end{figure}

  The condition that a Motzkin path does not go below the $x$-axis
  translates into the fact that $C_1$ and $C_2$ are non-intersecting.
  Since $P$ has $j$ steps weighted by $y$, the path $C_1$ ends at
  $(j,n-j+1)$.  Since $P$ ends at height $k$, the number of up steps
  $\nearrow$ and the number of level steps $\rightarrow$ with weight
  $y$ add up to $j+k$, so $C_2$ ends at $(j+k+1, n-j-k)$.

  By the Lindstr\"om-Gessel-Viennot Lemma~\cite{GeVi85}, these pairs
  of non-intersecting paths can be counted by a 
  $2\times 2$-determinant, which gives precisely Formula~\eqref{motz_pre_eq}.
\end{proof}

For $y=1$ the sum in Equation~\eqref{motz_pre_eq} can be simplified
using Vandermonde's identity.  Thus, the number of left factors of
Motzkin paths of length $n$ and final height $k$, with weight $2$ on
every level step is
\begin{equation} 
    \binom{2n}{n-k} - \binom{2n}{n-k-2}.
\end{equation}

For $(c,d)=(0,1)$, that is, $y=0$, we obtain what we need to count
modified crossings in set partitions, namely the binomial coefficient
$\binom{n}{k}$.

\section{counting $\mathcal{M}^*_k(a,b,c;q)$}
\label{sec:paths2}

In this section we use a continued fraction to find the generating
function for the Motzkin paths in $\mathcal{M}^*_k(a,b,c;q)$
(these paths are described in Definition \ref{defsMP}). It
turns out that this continued fraction can be expressed as a basic
hypergeometric series, which allows us to compute the coefficients
corresponding to paths with given length.  Let $K(a,b,c;q)$ be
\begin{equation}
   \cfrac{1}
     {1 + c - (a+b)   - \cfrac{(c-ab)(1-q)}
     {1 + c - (a+b)q  - \cfrac{(c-abq)(1-q^2)}
     {1 + c - (a+b)q^2- \cfrac{(c-abq^2)(1-q^3)} {\ddots}  } }}.
\end{equation}

Let us first give a combinatorial interpretation of $K(at,bt,ct^2;q)$
in terms of weighted Motzkin paths.  This result is close to those
given by Roblet and Viennot~\cite{RoVi96}, who
developed a combinatorial theory of $T$-fractions. These are continued
fractions of the form $1/(1-a_0t-b_0t/(1-a_1t-b_1t/\dots))$, and they 
are generating functions of Dyck paths with some weights on the peaks.

\begin{prop}\label{cfrac2}
  The coefficient of $t^k$ in the expansion of $K(at,bt,ct^2;q)$ is
  the generating function of $\mathcal{M}^*_k(a,b,c;q)$.
\end{prop}
\begin{proof} 
  The continued fraction $K(at,bt,ct^2;q)$ equals:
  \begin{equation}
    \cfrac{1}
    {1+ ct^2 - (a\hspace{-0.3mm}+\hspace{-0.3mm}b)t   - \cfrac{t^2(c - ab)(1-q)}
      {1+ ct^2 - (a\hspace{-0.3mm}+\hspace{-0.3mm}b)qt  - \cfrac{t^2(c - abq)(1-q^2)}
        {1+ ct^2 - (a\hspace{-0.3mm}+\hspace{-0.3mm}b)q^2t- 
          \cfrac{t^2(c\hspace{-0.3mm} -\hspace{-0.3mm} abq^2)
            (1\hspace{-0.3mm}-\hspace{-0.3mm}q^3)} {\ddots}  } }}.
  \end{equation}

  Using the ideas introduced by
  Flajolet~\cite[Theorem~1]{Fla82}, we thus obtain paths with four
  types of steps, denoted up $\nearrow$, down $\searrow$, level
  $\rightarrow$ and double-level $\longrightarrow$, the last type of
  step simply being twice as long as the usual level step.  Moreover,
  the weight of
  \begin{itemize}
  \item an up step $\nearrow$ starting at height $h$ is either $1$ or $-q^{h+1}$,
  \item a level step $\rightarrow$ starting at height $h$ is $(a+b)q^h$,
  \item a down step $\searrow$ starting at height $h$ is either $c$
    or $-abq^{h-1}$,
  \item a double-level step $\longrightarrow$ is $-c$.
  \end{itemize}
  To prove the statement, it suffices to construct a involution on
  the paths, such that
  \begin{itemize}
  \item its fixed points are precisely the elements of
    $\mathcal{M}^*_k(a,b,c;q)$, {\it i.e.} paths without double-level
    steps $\longrightarrow$ and without peaks $\nearrow\searrow$ such
    that the up step has weight $1$ and the down step has weight $c$,
  \item the weight of a path that is not fixed under the involution
    and the weight of its image add to zero.
  \end{itemize}
  Such an involution is easy to find: a path that is not in
  $\mathcal{M}^*_k(a,b,c;q)$, we look for the first occurrence of one
  of the two forbidden patterns, {\it i.e.} a double level step
  $\longrightarrow$ or a peak $\nearrow\searrow$ with steps weighted
  $1$ and $c$ respectively.  We then exchange one of the patterns for
  the other -- since the double level step $\longrightarrow$ has
  weight $-c$, the weights of the two paths add up to zero.
\end{proof}

As mentioned above, $K(at,bt,ct^2;q)$ can be expressed as a basic
hypergeometric series.  We use the usual notation for these series,
as for example in~\cite{GaRa90}.

\begin{prop} \label{cfrac1}
  For $A\neq 1$, $B\neq 0$, we have
  \renewcommand{\arraystretch}{1}
  \begin{equation} \label{cfrac_hyp}
    K(A,B,C;q) = 
    \frac{1}{1-A} \cdot {}_2\phi_1
    \left( \left. \begin{matrix} CB^{-1}q, q \\ Aq \end{matrix} \, \right| q,B \right). 
  \end{equation}
  For $A\neq 1$, $B=0$, we have
  \renewcommand{\arraystretch}{1}
  \begin{equation} \label{cfrac_hypb0}
    K(A,0,C;q) = 
    \frac{1}{1-A} \cdot {}_1\phi_1
    \left( \left. \begin{matrix} q \\ Aq \end{matrix} \, \right| q,Cq \right). 
  \end{equation}
\end{prop}
\begin{proof}
  Consider the following more general continued fraction, containing
  a new variable $z$:
  \[
  M(z) =
  \cfrac{1}
  {1+ C - (A+B)z   - \cfrac{(C - ABz)(1-qz)}
    {1+ C - (A+B)qz  - \cfrac{(C - ABqz)(1-q^2z)}
      {1+ C - (A+B)q^2z- \cfrac{(C - ABq^2z)(1-q^3z)} {\ddots}  } }}.
  \]
  Following Ismail and Libis~\cite{IsLi89} (see also
  Identity 19.2.11a in the Handbook of Continued Fractions for
  Special Functions~\cite{CPCWJ08}), we have:
  \renewcommand{\arraystretch}{1}
  \[
  M(z) = \frac{1}{1-z} \cdot
  {}_2\phi_1 \left( \left. \begin{matrix}  A,B \\ Cq  \end{matrix} \, \right| q,qz\right) \cdot
  {}_2\phi_1    \left(\left. \begin{matrix} A,B \\ Cq \end{matrix} \, \right| q,z \right)^{-1}.
  \]
  To be able to specialise $z=1$, we can use one of Heine's
  transformations \cite[p.13]{GaRa90}. For $B\neq 0$ we obtain
  \begin{align*}
    M(z) = & \;
    \frac{1}{1-z} \cdot
    \frac{ (Aqz,B,Cq,z;q)_\infty}{(Az,B,Cq,qz ;q)_\infty } \cdot
    {}_2\phi_1\left(\left. \begin{matrix} CB^{-1}q, qz \\ Aqz \end{matrix} \, \right| q,B \right)  \cdot
    {}_2\phi_1
    \left(\left. \begin{matrix} CB^{-1}q,z \\ Az \end{matrix} \, \right| q,B\right)^{-1} \\[3mm]
    = & \;
    \frac{1}{1-Az} \cdot
    {}_2\phi_1
    \left(\left. \begin{matrix} CB^{-1}q,qz \\ Aqz \end{matrix} \,\right| q,B \right) 
    \cdot {}_2\phi_1
    \left(\left. \begin{matrix} CB^{-1}q,z  \\ Az  \end{matrix} \,\right| q,B \right)^{-1}.
  \end{align*}
  In case $B=0$, we have
  \begin{align*}
    M(z) = & \;
    \frac{1}{1-z} \cdot
    \frac{ (Aqz,q,Cq,z;q)_\infty}{(Az,q,Cq,qz ;q)_\infty } \cdot
    {}_1\phi_1\left(\left. \begin{matrix} qz \\ Aqz \end{matrix} \, \right| q,Cq \right)  \cdot
    {}_1\phi_1
    \left(\left. \begin{matrix} z \\ Az \end{matrix} \, \right| q,Cq\right)^{-1} \\[3mm]
    = & \;
    \frac{1}{1-az} \cdot
    {}_1\phi_1
    \left(\left. \begin{matrix} qz \\ Aqz \end{matrix} \,\right| q,Cq \right) 
    \cdot {}_1\phi_1
    \left(\left. \begin{matrix} 1  \\ A  \end{matrix} \,\right| q,Cq \right)^{-1}.
  \end{align*}
\end{proof}

\begin{rem}
  Although the symmetry in $A$ and $B$ is not apparent in
  Equation~\eqref{cfrac_hyp}, it can be seen using one of Heine's
  transformations \cite[p.13]{GaRa90}.
\end{rem}

In the following, we will always use 
\renewcommand{\arraystretch}{1}
\begin{equation} \label{K_phi}
  K(at,bt,ct^2;q)=
   \frac{1}{1-at} \cdot  {}_2\phi_1
   \left( \left. \begin{matrix} cb^{-1}qt,q \\ aqt \end{matrix}\, \right| q,bt \right).
\end{equation}

Besides, it is also possible to use a method giving $M(z)$ as a
quotient of basic hypergeometric series without knowing {\it a
  priori} which identity to use.  This method was employed
in~\cite{CJPR09}, following Brak and Prellberg~\cite{PrBr95}.
Namely, note that the continued fraction expansion of $M(z)$ is
equivalent to the equation:
\begin{equation} \label{eq_fun}
 M(z) = \frac{1}{ 1 - c + (a+b)z - (c-abz)(1-qz)M(qz) }.
\end{equation}
By looking for solutions of the form $M(z) = (1-az)^{-1}
\frac{H(qz)}{H(z)} $, we obtain a linear equation in $H(z)$, which
gives a recurrence for the coefficients of the Taylor expansion of
$H(z)$, which is readily transformed into the explicit form of $H(z)$
as a basic hypergeometric series.

\subsection{Matchings}
\label{sec:matchings}
For matchings, we have $(a,b,c)=(0,0,1)$ and obtain:
\begin{lem}\label{qhermite} 
  \begin{equation*}
    K(0,0,t^2;q) = \sum_{k=0}^{\infty} (-t^2)^k q^{\binom{k+1}{2}}.
  \end{equation*}
\end{lem}
Essentially, this was shown by Penaud~\cite{Pen95}, who
enumerated $\mathcal{M}^*_{2k}(0,0,1;q)$ by first
constructing a bijection with parallelogram polyominoes, passing
through several intermediate objects with beautiful names like \lq
cherry trees\rq.  On the polyominoes he was finally able to construct
a weight-preserving, sign-reversing involution, with the only fixed
point having weight $(-1)^k q^{\binom{k+1}{2}}$, corresponding to
weighted Dyck paths with a single peak, and all weights maximal.

\subsection{Set partitions}
\label{sec:set-partitions}
For set partitions, we have $(a,b,c)=\big(0,-1,y(1-q)\big)$ and can
use the following lemma:
\begin{lem} \label{qchar}
  \begin{equation} \label{eq:qchar} 
    K(0,-t,ct^2;q) =
    \sum_{i=0}^{\infty} \sum_{j=0}^{i} t^{i+j} c^j (-1)^i
    q^{\binom{j+1}2} \qbin i j.
\end{equation}
\end{lem}
\begin{proof} 
  Using Equation~\eqref{K_phi} we find:
  \[
  K(0,-t,ct^2;q) =
  {}_2\phi_1
  \left(\left. \begin{matrix} -cqt,q \\ 0 \end{matrix} \,\right| q,-t \right) 
  = \sum_{i=0}^\infty (-cqt;q)_i (-t)^i.
  \]
  The proof follows by plugging in the elementary expansion
  \[
  (-cqt;q)_i = \prod_{j=1}^i (1+q^jct) = \sum_{j=0}^{i}
  q^{\binom{j+1}2} \qbin ij c^jt^j.
  \]
\end{proof}
In the appendix, we give a bijective proof of this lemma.

\subsection{Set partitions, modified crossings}
\label{sec:set-partitions-*}
When using the modified definition of crossings in set partitions, we
have $(a,b,c)=\big(-1,y(1-q),0\big)$ and can use the following lemma:
\begin{lem} \label{qchar*}
  \begin{equation} \label{qchar*_eq} K\big(-t,bt,0;q\big) =
    \sum_{i=0}^\infty\sum_{j=0}^i t^i b^j (-1)^{i-j} \qbin i j.
\end{equation}
\end{lem}
\begin{proof}
  Using Equation~\eqref{K_phi} we find:
  \[
  K(-t,bt,0;q) =\frac{1}{1+t} {}_2\phi_1 \left(\left. \begin{matrix}
        0,q \\ -qt \end{matrix} \,\right| q,bt \right) =
  \sum_{i=0}^\infty \frac{1}{(-t;q)_{i+1}} (bt)^i.
  \]
  The proof follows by plugging in the elementary expansion
  \[
  \frac{1}{(-t;q)_{i+1}} = \prod_{j=1}^i \frac{1}{1+q^jt} =
    \sum_{j\geq 0}
    \qbin{i+j}{j} (-t)^j.
  \]
\end{proof}

\subsection{Permutations}
\label{sec:permutations}
In the case of permutations, we have $(a,b,c)=(-1,-yq,y)$ and find:
\begin{lem} \label{An_K}
\[
   K(-t,-yqt,yt^2;q) = \sum_{k=0}^{\infty} (-t)^k 
   \left( \sum_{i=0}^k y^i q^{i(k+1-i)} \right).
\]
\end{lem}

\begin{proof} 
  Using Equation~\eqref{K_phi}, we have:
  \begin{align*}
    K(-t,-yqt,yt^2;q) = & \; 
    \frac{1}{1+t} \cdot  {}_2\phi_1
    \left(\left. \begin{matrix} -t,q \\ -qt \end{matrix} \, \right| q,-yqt \right) 
    = \sum_{i=0}^\infty  \frac{ (-yqt)^i }{ 1 + tq^i  } \\
    =& \; \sum_{i=0}^\infty \sum_{j=0}^\infty (-yqt)^i (-tq^i)^j
    =  \sum_{i=0}^\infty \sum_{j=0}^\infty (-t)^{i+j} y^i q^{i(j+1)}.
  \end{align*}
  To finish the proof it only remains to substitute $k$ for $i+j$.
\end{proof}

\section{Conclusion}

Let us briefly summarize how the four theorems announced in the introduction can
be proved using the previous sections. In each case, the enumeration of crossings
in combinatorial objects is linked with the enumeration of the weighted Motzkin
paths in $\mathcal{M}_n(a,b,c,d;q)$. The bijection $\Delta$ shows that the 
generating function of crossings can be decomposed into the generating functions
of the sets $\mathcal{P}_{n,k}(c,d)$ and $\mathcal{M}^*_n(a,b,c;q)$, which in turn
have been obtained in the previous two sections. This fulfills our initial objective 
as stated in the introduction.

\appendix

\section{Inverse relations}
We would like to mention an interesting non-bijective point of
view of the path decomposition given in Section \ref{sec:penaud-decomp}, 
using inverse relations.  Given
two sequences $\{a_n\}$ and $\{b_n\}$, an inverse relation is an
equivalence such as, for example:
\[
  \forall n\geq0, \quad a_n = \sum_{k=0}^n \binom n k b_k 
  \quad \Longleftrightarrow \quad
  \forall n\geq0, \quad b_n = \sum_{k=0}^n \binom n k (-1)^{n-k} a_k.
\]
This particular relation is easily proved by checking that the
(semi-infinite) lower  
triangular matrix $(\binom i j )_{i,j\in\mathbb{N}}$ has an inverse, which 
is $( (-1)^{i+j} \binom i j )_{i,j\in\mathbb{N}}$.  Other relations
of this kind can be found in Chapters~2 and~3 of Riordan's book
`Combinatorial Identities'~\cite{Rio68}.  To prove the Touchard-Riordan
formula~\eqref{eq:Touchard}, let 
$a_{2n}=(1-q)^n\mu_{2n}^{\tilde H}$ and $a_{2n+1}=0$.
We then use the following inverse relation~\cite[Chapter~2, Equation~(12)]{Rio68}:
\begin{equation} \label{inverse_pair1}
  a_n = \sum_{k=0}^{\lfloor \frac n2 \rfloor} 
  \left(\tbinom{n}{k} - \tbinom{n}{k-1}\right) b_{n-2k} 
  \quad \Longleftrightarrow \quad
  b_n = \sum_{k=0}^{\lfloor \frac n2 \rfloor} 
    (-1)^k \tbinom{n-k}{k} a_{n-2k}.
\end{equation}
Again, this can be proved by inverting a lower-triangular matrix.  It remains
to prove that $b_{2k}=(-1)^k q^{\binom{k+1}2}$ and $b_{2k+1}=0$.  To this end, we
relate the sequence $b_n$ to Schröder paths:

\begin{defi}
A {\it Schröder path} of length $2n$ is a path in $\mathbb{N}\times\mathbb{N}$ starting at 
$(0,0)$, arriving at $(2n,0)$ with steps $(1,1)$, $(1,-1)$, or $(2,0)$.
\end{defi}

\begin{lem}
Suppose that $a_{2n+1}=0$ and $a_{2n}$ is the generating function of Dyck paths of length
$2n$, with weight $1-q^{h+1}$ on each north-east step starting at height $h$ 
(this is to say $a_{2n}=(1-q)^n\mu_{2n}^{\tilde H}$). Suppose that $a_n$ and $b_n$ are related
by \eqref{inverse_pair1}. Then $b_{2n+1}=0$, and $b_{2n}$ is the generating function of Schröder
paths of length $2n$, with weight $-1$ on each level step, and $1-q^{h+1}$ on each north-east 
step starting at height $h$.
\end{lem}

\begin{proof}
  For any even $n$, consider a Schröder path of length $n$ with $k$ level
  steps, weighted as described above.  This path has $n-k$ non-level
  steps, and can thus be obtained from a Dyck path of length $n-2k$
  by inserting the level steps.  There are $\binom {n-k}{k}$ ways to
  do so, which implies that the generating function of Schröder paths
  is indeed equal to $\sum_{k=0}^{\lfloor \frac n2 \rfloor} (-1)^k
  \tbinom{n-k}{k} a_{n-2k}$, and therefore equal to $b_n$.
\end{proof}

With the bijective decomposition of paths in Section
\ref{sec:penaud-decomp}, we showed that we have to obtain the
generating function of the set $\mathcal{M}^*_k(0,0,1;q)$ to prove
Touchard-Riordan formula. With the inverse relations, we showed that
we have to count certain weighted Schröder paths.  The fact that both
sets have the same generating function follows from the involution
given in the proof of Proposition~\ref{cfrac2}.

It is also possible to use an inverse relation to obtain the formula
in Theorem~\ref{th:Laguerre} for the $q$-Laguerre moments.
By inverting a lower triangular matrix, one can check that
\begin{equation}
  a_n = \sum_{k=0}^n \left( \sum_{j=0}^k y^j \Big( \tbinom{n}{j}\tbinom{n}{j+k} -
        \tbinom{n}{j-1}\tbinom{n}{j+k+1}\Big) \right) b_k
\end{equation}
for all $n$ is equivalent to 
\begin{equation} \label{inv22}
  b_n = \sum_{k=0}^n \left( \sum_{j=0}^{\lfloor \frac{n-k}2\rfloor} \binom{n-j}{n-k-j}
  \binom{n-k-j}j (-y)^j (-1-y)^{n-k-2j} \right) a_k
\end{equation}
for all $n$.  

Equation~\eqref{inv22} can be interpreted as follows: given that
$a_k$ counts elements of $\mathcal{M}_n(1,-yq,y,1+y;q)$, then $b_n$
count paths of length $n$, with the same weights as in
$\mathcal{M}_n(1,-yq,y,1+y;q)$, where we insert some level steps
$\rightarrow$ with weight $-1-y$, and some double level steps
$\longrightarrow$ with weight $-y$.

Indeed, suppose that we inserted $j$ double-level steps $\longrightarrow$, 
and hence $n-k-2j$ level steps $\rightarrow$, starting with a path of length $k$.  
This yields the weight $(-y)^j (-1-y)^{n-k-2j}$ for the inserted
steps.  The first binomial coefficient,
$\binom{n-j}{n-k-j}$, is the number of ways to insert the $n-k-j$ level steps among the 
$n-j$ steps (the total number of steps being $n-j$ because the length
is $n$, and $j$ steps of double length). 

The second binomial coefficient, $\binom{n-k-j}{j}$, is the number of combinations of the 
$j$ inserted double-level steps and the $n-k-2j$ inserted level steps.
Using the involution given in the proof of Proposition \ref{cfrac2}, we see
that $b_n$ counts elements in the set $\mathcal{M}_n^*(1,-yq,y;q)$.

\section{A bijective proof of Lemma~\ref{qchar}}
\label{sec:biject-proof-lemma}

We show in this appendix that Penaud's bijective method of proving
Lemma~\ref{qhermite} can be generalised to prove also
Lemma~\ref{qchar}.  Namely, we construct a sign-reversing involution
on a set of weighted Motzkin paths, whose fixed points are enumerated
by the right-hand side of \eqref{eq:qchar}.  This involution was
essentially given by the first author in \cite{MJV} in a different
context.  We take the opportunity to correct some mistakes in 
this reference.

In the following we fix integers $j,k\geq0$ and consider the set
$\mathcal{C}_{j,k}$ of Motzkin paths of length $k+j$ with $k-j$ level
steps $\rightarrow$, (and hence $j$ up steps $\nearrow$ and $j$ down
steps $\searrow$), satisfying the following conditions:
\begin{itemize}
\item the weight of all up steps $\nearrow$ is $1$,
\item the weight of a level step $\rightarrow$ at height $h$ is $-q^h$,
\item the weight of a down step $\searrow$ starting at height $h$ is
  either $1$ or $-q^h$,
\item there is no peak $\nearrow\searrow$ such that both the up step
  and the down step have weight $1$.
\end{itemize}
The generating function of $\mathcal{C}_{j,k}$ is then the
coefficient of $a^kt^{k+j}$ in $K(0,-t,at^2;q)$.  It thus suffices
to prove the following:
\begin{prop}\label{invol}
  There is an involution $\theta$ on the set $\mathcal{C}_{j,k}$ such
  that:
  \begin{itemize}
  \item the fixed points are the paths that
    \begin{itemize}
    \item start with $j$ up steps $\nearrow$,
    \item and contain no down steps $\searrow$ of weight $1$,
    \end{itemize}
  \item the weight of a path that is not fixed under the involution
    and the weight of its image add to zero.
  \end{itemize}
  Moreover, the sum of weights of fixed points of the involution is
  $(-1)^kq^{\binom{j+1}{2}}\qbin kj$.
\end{prop}

Penaud's method consists in introducing several intermediate objects as 
described in Section \ref{sec:matchings}. However, in the case at hand we will 
not use intermediate objects, but rather construct the involution directly on
the paths. What we give is a generalisation of Penaud's construction, 
which we recover in the case $k=j$. 

\begin{proof} 
Following Penaud~\cite{Pen95}, we use in this proof a word notation for elements in 
$\mathcal{C}_{j,k}$. The letters $x$, $z$, $y$, and $\bar{y}$ will respectively denote
the steps $\nearrow$,  $\rightarrow$, $\searrow$ with weight $1$, and $\searrow$ with 
weight $-q^h$. For any word $c\in\mathcal{C}_{j,k}$, we define:
\begin{itemize}
\item $u(c)$ as the length of the last sequence of consecutive $x$,
\item $v(c)$ as the starting height of the last step $y$, if $c$ contains a $y$
  and there is no $x$ after the last $y$, and $j$ otherwise.
\end{itemize}
See Figure \ref{invol_cores} for an example. The fixed points of  $\theta$ will be
$c\in\mathcal{C}_{j,k}$ such that $u(c)=v(c)=j$, which correspond to the paths described
in Proposition \ref{invol}.

\smallskip

now, suppose $c$ is such that $u(c)<j$ or $v(c)<j$. We will build
$\theta$ so that $v(c)\leq u(c)$ if and only if $u(\theta(c)) < v(\theta(c))$.
Thus it suffices to define $\theta(c)$ in the case $v(c)\leq u(c)$, and to check
that we have indeed $u(\theta(c)) < v(\theta(c))$. So let us suppose $v(c)\leq u(c)$, 
hence $v(c)<j$.

\smallskip

Since $v(c)<j$, there is at least a letter $y$ in $c$ 
having no $x$ to its right. Let $\tilde{c}$ be the word obtained from $c$ 
by replacing the last $y$ with a $\bar{y}$. There is a unique factorisation
 \[ \tilde{c}=f_1x^{u(c)} ay^{\ell}f_2 \] 
such that:
\begin{itemize}
\item $a$ is either $z$, or $\bar{y}$,
\item $f_2$ begins with $z$ or $\bar{y}$, contains at least one
  letter $\bar{y}$, but contains no $x$.
\end{itemize}
Let us explain this factorisation. By definition of $u(c)$, we can write 
$\tilde{c}=f_1x^{u(c)} c'$, where $c'$ does not contain any $x$. In a word
$c\in\mathcal{C}_{j,k}$, an $x$ cannot be followed by a $y$. So we can write
$\tilde{c}=f_1x^{u(c)} a c''$ where $a$ is either $z$, or $\bar{y}$. Then, we write
$c''=y^\ell f_2$ with $\ell\geq0$ maximal, and $f_2$ satisfy the conditions
($f_2$ contains indeed a $\bar y$ because we transformed a $y$ into a $\bar y$). 
Uniqueness is immediate.

\smallskip

We set:
\begin{equation} \label{def_theta}
  \theta(c) =  f_1x^{u(c)-v(c)}ay^\ell x^{v(c)}f_2.
\end{equation}
See Figure \ref{invol_cores} for an example with $u(c)=4$, $v(c)=2$, $j=9$ and $k=12$.
We can check that $w(c)=-q^{19}= - w(\theta(c))$, and $u(\theta(c))=2$, $v(\theta(c))=3$.

\begin{figure}[h!tp]  \center\psset{unit=3mm}
\begin{pspicture}(0,-1)(21,8)
\psgrid[gridcolor=gray,griddots=4,subgriddiv=0,gridlabels=0](0,0)(21,6)
\psline(0,0)(1,1)(2,2)(3,1)(4,2)(5,2)(6,1)(7,2)(8,3)(9,2)(10,3)(11,4)(12,5)(13,6)
(14,6)(15,5)(16,4)(17,3)(18,2)(19,1)(20,1)(21,0)
\psline[linewidth=0.8mm]{c-c}(2,2)(3,1)  \psline[linewidth=0.8mm]{c-c}(8,3)(9,2)
\psline[linewidth=0.8mm]{c-c}(16,4)(17,3)\psline[linewidth=0.8mm]{c-c}(20,1)(21,0)
\psline[linewidth=0.8mm]{c-c}(4,2)(5,2)  \psline[linewidth=0.8mm]{c-c}(19,1)(20,1)
\psline[linewidth=0.8mm]{c-c}(13,6)(14,6) \rput(-2,3){$c=$}
\rput(0.5,-1){$x$}\rput(1.5,-1){$x$}\rput(2.5,-1){$\bar y$}\rput(3.5,-1){$x$}\rput(4.5,-1){$z$}
\rput(5.5,-1){$y$}\rput(6.5,-1){$x$}\rput(7.5,-1){$x$}\rput(8.5,-1){$\bar y$}\rput(9.5,-1){$x$}
\rput(10.5,-1){$x$}\rput(11.5,-1){$x$}\rput(12.5,-1){$x$}\rput(13.5,-1){$z$}
\rput(14.5,-1){$y$}\rput(15.5,-1){$y$}\rput(16.5,-1){$\bar y$}\rput(17.5,-1){$y$}
\rput(18.5,-1){$y$}\rput(19.5,-1){$z$}\rput(20.5,-1){$\bar y$}
\psline{<->}(12.9,6.5)(9.1,6.5)
\rput(11.5,7.3){$u(c)=4$}
\rput(23,3.2){$v(c)=2$}\psline{->}(21.5,2.8)(18.6,1.6)
\end{pspicture} \\[1.4cm]
\begin{pspicture}(0,-1)(21,8)
\psgrid[gridcolor=gray,griddots=4,subgriddiv=0,gridlabels=0](0,0)(21,6)
\psline(0,0)(1,1)(2,2)(3,1)(4,2)(5,2)(6,1)(7,2)(8,3)(9,2)(10,3)(11,4)(12,5)(13,6)
(14,6)(15,5)(16,4)(17,3)(18,2)(19,1)(20,1)(21,0)
\psline[linewidth=0.8mm]{c-c}(2,2)(3,1)  \psline[linewidth=0.8mm]{c-c}(8,3)(9,2)
\psline[linewidth=0.8mm]{c-c}(16,4)(17,3)\psline[linewidth=0.8mm]{c-c}(20,1)(21,0)
\psline[linewidth=0.8mm]{c-c}(4,2)(5,2)  \psline[linewidth=0.8mm]{c-c}(19,1)(20,1)
\psline[linewidth=0.8mm]{c-c}(13,6)(14,6) \psline[linewidth=0.8mm]{c-c}(18,2)(19,1)
\rput(-2,3){$\tilde c=$}
\rput(0.5,-1){$x$}\rput(1.5,-1){$x$}\rput(2.5,-1){$\bar y$}\rput(3.5,-1){$x$}\rput(4.5,-1){$z$}
\rput(5.5,-1){$y$}\rput(6.5,-1){$x$}\rput(7.5,-1){$x$}\rput(8.5,-1){$\bar y$}\rput(9.5,-1){$x$}
\rput(10.5,-1){$x$}\rput(11.5,-1){$x$}\rput(12.5,-1){$x$}\rput(13.5,-1){$z$}
\rput(14.5,-1){$y$}\rput(15.5,-1){$y$}\rput(16.5,-1){$\bar y$}\rput(17.5,-1){$y$}
\rput(18.5,-1){$\bar y$}\rput(19.5,-1){$z$}\rput(20.5,-1){$\bar y$}
\psline{<->}(8.9,6.5)(0.1,6.5)
\psline{<->}(12.9,6.5)(9.1,6.5)
\psline{<->}(15.9,6.5)(13.1,6.5)
\psline{<->}(20.9,6.5)(16.1,6.5)
\rput(5,7.3){$f_1$}
\rput(11.5,7.3){$x^{u(c)}$}
\rput(14.5,7.3){$ay^\ell $}
\rput(18,7.3){$f_2$}
\end{pspicture} \\[1.4cm]
\begin{pspicture}(0,-2)(21,8)
\psgrid[gridcolor=gray,griddots=4,subgriddiv=0,gridlabels=0](0,0)(21,6)
\psline(0,0)(1,1)(2,2)(3,1)(4,2)(5,2)(6,1)(7,2)(8,3)(9,2)(10,3)(11,4)(12,4)(13,3)
(14,2)(15,3)(16,4)(17,3)(18,2)(19,1)(20,1)(21,0)
\psline[linewidth=0.8mm]{c-c}(2,2)(3,1)  \psline[linewidth=0.8mm]{c-c}(8,3)(9,2)
\psline[linewidth=0.8mm]{c-c}(16,4)(17,3)\psline[linewidth=0.8mm]{c-c}(18,2)(19,1)
\psline[linewidth=0.8mm]{c-c}(20,1)(21,0)
\psline[linewidth=0.8mm]{c-c}(4,2)(5,2)  \psline[linewidth=0.8mm]{c-c}(19,1)(20,1)
\psline[linewidth=0.8mm]{c-c}(11,4)(12,4) \rput(-2,3){$\theta(c)=$}
\rput(0.5,-1){$x$}\rput(1.5,-1){$x$}\rput(2.5,-1){$\bar y$}\rput(3.5,-1){$x$}\rput(4.5,-1){$z$}
\rput(5.5,-1){$y$}\rput(6.5,-1){$x$}\rput(7.5,-1){$x$}\rput(8.5,-1){$\bar y$}\rput(9.5,-1){$x$}
\rput(10.5,-1){$x$}\rput(11.5,-1){$x$}\rput(12.5,-1){$x$}\rput(13.5,-1){$z$}
\rput(14.5,-1){$y$}\rput(15.5,-1){$y$}\rput(16.5,-1){$\bar y$}\rput(17.5,-1){$y$}
\rput(18.5,-1){$\bar y$}\rput(19.5,-1){$z$}\rput(20.5,-1){$\bar y$}
\psline{<->}(8.9,6.5)(0.1,6.5)
\psline{<->}(10.9,6.5)(9.1,6.5)
\psline{<->}(13.9,6.5)(11.1,6.5)
\psline{<->}(15.9,6.5)(14.1,6.5)
\psline{<->}(20.9,6.5)(16.1,6.5)
\rput(4,7.3){$f_1$}
\rput(9,7.6){$x^{u(c)-v(c)}$}
\rput(12.8,7.3){$ay^\ell $}
\rput(15.3,7.6){$x^{v(c)}$}
\rput(19,7.3){$f_2$}
\end{pspicture}
\mycaption{ An element $c$ in $\mathcal{C}_{j,k}$, 
and its image $\theta(c)$. Thick steps are the steps $\bar{y}$ and $z$, 
{\it i.e.} steps with weight $-q^h$. 
\label{invol_cores}}
\end{figure}

\noindent
We show the following points:
\begin{itemize}
\item The path $\theta(c)$ is a Motzkin path. Indeed, the factor $ay^j$ in
  $c$ ends at height at least $v(c)$, since the factor $f_2$ contains a step $\bar y$ 
  starting at this height and contains no $x$. We can thus shift this factor
  $\tilde c$ so that the result is again a Motzkin path.
\item The path $c$ and its image $\theta(c)$ have opposite weights. To begin, 
  between $c$ and $\tilde c$, the weight is multiplied by $-q^{v(c)}$, 
  since we have transformed a $y$ into a $\bar y$ starting at height $v(c)$. 
  Between $\tilde c$ and $\theta(c)$, the height of the factor $ay^j$ has
  decreased by $v(c)$, so the weight has been divided by $q^{v(c)}$.
  A factor $-1$ remains, which proves the claim.
\item The path $\theta(c)$ is such that $u(\theta(c))<v(\theta(c))$. From the 
  definition (\ref{def_theta}) we see that $u(\theta(c))=v(c)$. Besides, 
  $v(c) < v(\theta(c))$ since the last step $y$ of $c$ has been transformed into 
  a $\bar{y}$ to obtain $\tilde c$ and $\theta(c)$.
\item Every path $c'$ with $u(c')<v(c')$ is obtained as a $\theta(c)$ for some other path
  $c$ with $u(c)\geq v(c)$. Indeed, let $\tilde{c}'$ be the word obtained from $c'$ after 
  replacing the last $\bar{y}$ at height $u(c')$ with a $y$. There is a unique factorisation
  $\tilde{c}'=f_1ay^jx^{u(c')}f_2$, where $a$ is $z$ or $\bar{y}$, and $f_2$ contains no
  $x$. Then by construction, $c=f_1x^{u(c')}ay^jf_2$ has the required properties.
\end{itemize}

\medskip

Thus, $\theta$ is indeed an involution with the announced fixed points.

\medskip

It remains only to check that the sum of weights of the fixed points is equal to
$(-1)^kq^{\binom{j+1}{2}}\qbin kj$. A fixed point of $\theta$ is specified by the 
heights $h_1,\dots,h_{k-j}$ of the  $k-j$ steps $\rightarrow$. These heights 
can be any set of values provided that $j\geq h_1 \geq \dots \geq h_{k-j} \geq 0$.
The weight of such a fixed point is:
\[  
  (-1)^jq^{\frac{j(j+1)}2} \prod_{i=0}^{k-j} (-q^{h_i}) 
  =  (-1)^kq^{\frac{j(j+1)}2} q^{\sum_{i=0}^{k-j}h_i}.
\]
Indeed, the $j$ steps $\nearrow$ have respective weights $-q,-q^2,\dots,-q^j$, which gives
a factor $(-1)^jq^{\frac{j(j+1)}2}$. Besides, we have:
\[
        \sum_{j\geq h_1\geq\dots\geq h_{k-j}\geq 0} q^{\sum h_i} = \qbin kj,
\]
by elementary property of $q$-binomial coefficients. This ends the proof.
\end{proof}

\section{Closed forms for $\size{\mathcal{P}_{n,k}(c,d)}$}
\label{sec:trinomial}

In this appendix we give a justification of the fact that there is no
(hypergeometric) closed form (in the sense of Petkov{\v{s}}ek, Wilf
and Zeilberger \cite{AeqB}) for
\begin{equation}
  \size{\mathcal{P}_{n,k}(c,d)}=
  \frac{k+1}{n+1}\sum_{l=0}^{n-k}\binom{n+1}{l}\binom{l}{2l-n+k}d^{2l-n+k}c^{n-k-l}.
\end{equation}
except when $(c,d)$ is one of $(1,0)$, $(0,1)$ or $(y^2,2y)$.  More
precisely, we claim that $\size{\mathcal{P}_{n,k}(c,d)}$ cannot be
written as a linear combination (of a fixed finite number) of
hypergeometric terms except in the specified cases.  In the following
we sketch a straightforward way to check this is using computer
algebra.

First we convert the summation into a polynomial recurrence
equation.  This can be done by using Zeilberger's algorithm (which
also proves that the recurrence is correct), for example.  Writing
$p_n = \size{\mathcal{P}_{n,k}(c,d)}$ we obtain
\begin{equation}
  (4c-d^2)(n+1)(n+2)p_n + d(n+2)(2n+5)p_{n+1}-(n+2-k)(n+4+k)p_{n+2}=0.
\end{equation}
Alternatively, one can also find a recurrence for $q_k =
\size{\mathcal{P}_{n,k}(c,d)}$, which is
\begin{equation}
  (k+3)(k-n)q_k + d(k+1)(k+3) q_{k+1} + c(k+1)(k+n+4)q_{k+2}=0.
\end{equation}

It now remains to show that both equations admit no hypergeometric
solutions, except for the values of $(c,d)$ mentioned above.  To this
end we use Petkov{\v{s}}ek's algorithm {\bf hyper}, as described in
Chapter~8 of `A=B'~\cite{AeqB}.  Unfortunately, this time we
cannot use the implementation naively.  Namely, {\it a priori} {\bf
  hyper} decides only for fixed parameters $(c,d)$ whether a
hypergeometric solutions exists or not.

However, it is possible to trace the algorithm, and, whenever it has
to decide whether a quantity containing $c$ or $d$ is zero or not, do
it for the computer.  (Of course, it should be possible to actually
program this, but that is outside the scope of this article.)  We
refrain from giving a complete proof, but rather give only a few
details to make checking easier.

First of all, let us assume that $c$, $d$ and $d^2-4c$ are all
nonzero.  Then the degrees of the coefficient polynomials in both
recurrence equations are all the same.  From now on, the procedure is
the same for both recurrence equation, so let us focus on the one for
$p_n$.  According to the remark in Example~8.4.2 in~\cite{AeqB}, we
have to consider all monic factors $a(n)$ of the coefficient of
$p_n$, and also the monic factors $b(n)$ of the coefficient of
$p_{n+2}$, such that the degree of $a(n)$ and $b(n)$ coincide.  In
this case, the characteristic equation, Equation~(8.4.5)
in~\cite{AeqB} one has to solve turns out to be $z^2-2dz+d^2-4c$.
For each of the two solutions in $z$, one has to check that there is
no polynomial solution of the recurrence
$$P_0(n)c_n + zP_1(n)c_{n+1} + z^2P_2(n) c_{n+2}=0,$$
where the coefficient polynomials $P_0(n),P_1(n)$ and $P_2(n)$ are
polynomials derived from the coefficient polynomials of the original
recurrence by multiplying with certain shifts of $a(n)$ and $b(n)$.

This can be done with the algorithm {\bf poly}, described in
Section~8.3 of~\cite{AeqB}.  Namely, depending on the degrees of yet
another set of polynomials derived from $P_0(n),P_1(n)$ and $P_2(n)$,
it computes an upper bound for the degree of a possible polynomial
solution.  Indeed, the algorithm decides that the degree of such a
solution would have to be negative, provided that $c$, $d$ and
$d^2-4c$ are nonzero, which is what we assumed.


\begin{thebibliography}{999}



\bibitem{Cor07}
 S. Corteel, Crossings and alignments of permutations, Adv. in App. Math.
 38(2) (2007), 149--163.

\bibitem{CJPR09}
 S. Corteel, M. Josuat-Vergès, T. Prellberg and R. Rubey,
 Matrix Ansatz, lattice paths and rook placements, Proc. FPSAC 2009.

\bibitem{CPCWJ08}
 A. Cuyt, A.B. Petersen, B. Verdonk, H. Waadeland and W. B. Jones, Handbook of Continued 
 Fractions for Special Functions, Springer, 2008. 

\bibitem{MSW}
 A. de Médicis, D. Stanton and D. White,
 The combinatorics of $q$-Charlier polynomials, 
 J. Combin. Theory Ser. A 69 (1995), 87--114.

\bibitem{ER96}
 R. Ehrenborg and M. Readdy,
 Juggling and applications to $q$-analogues, Discrete Math. 157 (1996), 107--125

\bibitem{Fla82}
 P. Flajolet, Combinatorial aspects of continued fractions, Discrete Math. 41 (1982), 145--153.

\bibitem{GaRa90}
 G. Gasper and M. Rahman, Basic hypergeometric series, Second edition, 
 Cambridge University Press, 2004.

\bibitem{GeVi85}
 I.M. Gessel and X.G. Viennot, Binomial determinants, paths and hook length formulae, 
 Adv. in Math. 58 (1985), 300--321.

\bibitem{Gou61}
 H.W. Gould, The $q$-Stirling numbers of first and second kinds, 
 Duke Math. J. 28(2) (1961), 281--289.

\bibitem{IsLi89}
 M.E.H. Ismail and C. A. Libis, Contiguous relations, basic
 hypergeometric functions, and orthogonal polynomials, I. J.
 Math. Anal. Appl. 141(2) (1989), 349--372.

\bibitem{ISV87}
 M.E.H. Ismail, D. Stanton and X.G. Viennot, The combinatorics of $q$-Hermite polynomials and the
 Askey-Wilson integral, European J. Combin. 8 (1987), 379--392.

\bibitem{MJV}
 M. Josuat-Vergès, Rook placements in Young diagrams and permutation enumeration,
 to appear in Adv. Appl. Math. (2010).

\bibitem{KSZ08}
 A. Kasraoui, D. Stanton and J. Zeng, The combinatorics of Al-Salam-Chihara 
 $q$-Laguerre polynomials, arXiv$:$0810.3232v1 [math.CO].

\bibitem{KaZe06}
 A. Kasraoui and J. Zeng, Distribution of crossings, nestings and alignments of two edges in 
 matchings and partitions, Electron. J. Combin.  13(1)  (2006), R33.

\bibitem{KSZ06}
 D. Kim, D. Stanton and J. Zeng,
 The combinatorics of the Al-Salam-Chihara $q$-Charlier polynomials,
 Sém. Lothar. Combin. 54 (2006), Article B54i.

\bibitem{KoSw98}
 R. Koekoek, P.A. Lesky and R.F. Swarttouw,
 Hypergeometric Orthogonal Polynomials and Their $q$-analogues,
 Springer, 2010.

\bibitem{Pen95} 
 J.-G. Penaud, A bijective proof of a Touchard-Riordan formula, Discrete Math. 139 (1995), 347--360.

\bibitem{AeqB} 
  M. Petkov{\v{s}}ek, H. Wilf and D. Zeilberger, $A=B$, Peters
  Ltd. 1996.

\bibitem{PrBr95}
 T. Prellberg and R. Brak, Critical Exponents from Non-Linear Functional Equations for Partially 
 Directed Cluster Models, J. Stat. Phys. 78 (1995), 701--730.


\bibitem{Rio75}
 J. Riordan, The distribution of crossings of chords joining pairs of 2n points on a circle, 
 Math. Comput. 29(129) (1975), 215--222.

\bibitem{Rio68}
 J. Riordan, Combinatorial identities, Wiley, 1968.

\bibitem{RoVi96}
 E. Roblet and X.G. Viennot, Théorie combinatoire des T-fractions et approximants 
 de Padé en deux points, Discrete Math. 153(1-3) (1996), 271--288.

\bibitem{SS}
 R. Simion and D. Stanton,
 Octabasic Laguerre polynomials and permutation statistics, 
 J. Comp. Appl. Math. 68 (1996), 297--329.

\bibitem{Sta}
 R. Stanley, Enumerative Combinatorics Vol. 2, Cambridge University Press, 1999.

\bibitem{Tou52}
 J. Touchard, Sur un problème de configurations et sur les fractions continues,
 Can. J. Math. 4 (1952), 2--25.

\bibitem{Vie84}
 X.G. Viennot, Une théorie combinatoire des polynômes 
 orthogonaux, Lecture notes, UQÀM, Montréal, 1984.

\end{thebibliography}

\section*{Acknowledgement}

This article develops ideas presented in the extended abstract
\cite{CJPR09}.  We thank Sylvie Corteel and Thomas Prellberg, who
contributed considerably to our curiosity about generalising Penaud's
construction.

\end{document}